\documentclass[11pt, reqno, dvipsnames]{amsart}
\usepackage[utf8]{inputenc}
\usepackage[english]{babel}
\usepackage{color}
\usepackage{amsmath,amsthm,amssymb,amsfonts}
\usepackage{mathrsfs}
\usepackage{tikz-cd}
\usepackage[linktocpage=true]{hyperref}
\usepackage{enumerate}
\usepackage{geometry}
\usepackage{comment}

\hypersetup{
	colorlinks = true,
	linkcolor = {ForestGreen},
	citecolor = {RoyalBlue},
	urlcolor = {Black}
}

\newcommand{\Z}{{\mathbb Z}}
\newcommand{\Q}{{\mathbb Q}}
\newcommand{\R}{{\mathbb R}}

\newcommand{\N}{{\mathbb N}}
\newcommand{\F}{{\mathbb F}}

\newcommand{\f}{{\mathcal F}}

\newcommand{\Spec}{\operatorname{Spec}}

\newcommand{\Spa}{\operatorname{Spa}}

\newcommand{\an}{\operatorname{an}}

\newcommand{\et}{\operatorname{\acute{e}t}}
\newcommand{\ket}{\operatorname{k\acute{e}t}}

\newcommand{\proket}{\operatorname{prok\acute{e}t}}

\numberwithin{equation}{section}
\newtheorem{theorem}{Theorem}
\numberwithin{theorem}{section}

\newtheorem{thm}[theorem]{Theorem}
\newtheorem{lem}[theorem]{Lemma}
\newtheorem{cor}[theorem]{Corollary}
\newtheorem{prop}[theorem]{Proposition}

\theoremstyle{definition}
\newtheorem{defn}[theorem]{Definition}
\newtheorem{rem}[theorem]{Remark}
\newtheorem{ex}[theorem]{Example}

\AtBeginDocument{%
	\def\MR#1{}
}

\numberwithin{equation}{section}
\begin{document}
	
	\title{Log blow-ups and global resolutions in rigid geometry}

	\author{Xinyu Shao}

	\address{Institut für Algebra, Zahlentheorie und Diskrete Mathematik, Leibniz Universität Hannover, Welfengarten 1, 30167 Hannover, Germany}
	
	\email{xinyu.shao@math.uni-hannover.de}
	
	\begin{abstract}
		In this article, we study log blow-ups in the setting of log adic spaces. As in the theory of log schemes, we prove the existence of global resolutions of toric singularities for fs log adic spaces log smooth over a $p$-adic field, and the acyclicity of certain sheaves under log blow-ups. One new feature of (log) rigid analytic varieties is their integral structure sheaf: we prove that it is also almost acyclic under log blow-ups in the Kummer \'etale topology. We also supply some foundational results on log adic spaces which are not covered in \cite{dllz2023logadic}.
	\end{abstract}
	
	\maketitle
	\tableofcontents
	
	\section{Introduction}

	Let $K$ be a complete extension of $\Q_p$. Let $\overline{K}$ be an algebraic closure of $K$, and let $C$ be its $p$-adic completion. The main goal of this article is to provide the necessary background in logarithmic geometry for the author’s forthcoming work \cite{xslogsyn} on logarithmic syntomic cohomology.

	Logarithmic geometry, based on the notion of log structures due to Fontaine, Illusie and Kato \cite{kato1989log}, was developed to treat certain singular schemes, such as semistable degenerations and varieties with normal crossing boundaries, as if they were smooth. Since then, it has been well studied in the algebraic setting. On the cohomological side, the log crystalline cohomology of Hyodo and Kato \cite{hyodokato1994} (see also \cite{tsuji1999poincare}) plays an important role in $p$-adic Hodge theory, for example in the proof of the semistable comparison theorem \cite{tsuji1999padicetale}. The Kummer \'etale topology, introduced by Kato and further studied by Fujiwara, Kato and Nakayama \cite{nakayama1997logetale}, provides a good theory of \'etale cohomology for log schemes. On the geometric side, Kato \cite{kato1994toric} studied log regular schemes, whose singularities are called toric singularities, together with their log blow-ups, and Nizio\l{} \cite{niziol2006toric} proved that toric singularities can be resolved globally by log blow-ups. 
	
	These two sides are closely related: many cohomology theories, such as the Kummer \'etale cohomology of locally constant sheaves and coherent cohomology, are invariant under log blow-ups. Together with the resolution of toric singularities, this often allows one to reduce problems on log regular schemes to the case where the log structure comes from a strictly normal crossing divisor. We refer the reader to \cite{illusie2002log} for an overview of logarithmic \'etale cohomology, which also includes results from unpublished papers, and to \cite{ogus2018log} for a systematic treatment of logarithmic geometry.

	In the rigid analytic setting, the foundations of log adic spaces were developed by Hansheng Diao, Kai-Wen Lan, Ruochuan Liu and Xinwen Zhu in \cite{dllz2023logadic}. They used these foundations to prove many results: for example, they constructed a candidate Riemann--Hilbert functor for algebraic varieties in \cite{dllz2023logrh}, where the \'etale topology on the analytic side (not just on the algebraic side) is necessary. In \cite{llz2023derham}, Lan, Liu and Zhu studied \'etale cohomology with compact support and proved Poincar\'e duality for smooth and almost proper rigid analytic varieties (i.e., varieties of the form $X\setminus Z$, where $X$ is proper and $Z$ is Zariski closed in $X$).

	Following the theory of logarithmic geometry developed in the algebraic setting, we extend the foundations of \cite{dllz2023logadic} in the rigid analytic setting. This article has two main goals. Firstly, we define log blow-ups of log adic spaces, study their basic properties, and prove the acyclicity of certain sheaves under log blow-ups. This includes the classical cases, such as locally constant sheaves and the structure sheaf. We also study the integral structure sheaf of log adic spaces, which is a new feature of the rigid analytic setting: we prove that $\mathcal O^+/p$ and $\mathcal O^+$ are almost acyclic under log blow-ups in the Kummer \'etale topology\footnote{However, this is not true in the \'etale topology; see Remark \ref{notacyclic}.}. The theorem is stated as follows.
	
	\begin{theorem} (Theorem \ref{acyclicintegral} and \ref{acyclicintegral2})
		Let $X$ be an analytic fs log rigid variety over $C$, and let $f:Y\to X$ be a log blow-up along a coherent ideal $\mathcal{J}\subset\mathcal{M}_X$ whose geometric stalks are nonempty. Then the adjunction map $$\mathcal{O}_X^{+}/p \to Rf_*(\mathcal{O}_Y^{+}/p)$$ is an almost quasi-isomorphism over the Kummer \'etale site. Moreover, when $X$ is log smooth over $C$,  the adjunction map $$\mathcal{O}_X^{+} \to Rf_*(\mathcal{O}_Y^{+})$$ is an almost quasi-isomorphism over $X_{\ket}$.
	\end{theorem}
	
	Secondly, we prove that the toric singularities of a quasi-compact fs log adic space log smooth over $K$ can be resolved globally by a sequence of log blow-ups. This is an analogue of the main result of Nizio\l{} \cite{niziol2006toric} in the rigid analytic setting. The theorem is stated as follows.

	\begin{thm} (Theorem \ref{globalresolution})
		Let $X$ be a quasi-compact fs log rigid variety log smooth over $K$, and let $X^\circ$ be the maximal open subspace of $X$ on which the log structure is trivial. Then there exists a morphism $\pi:X'\to X$ which is a composition of log blow-ups, such that $X'$ is smooth over $K$, $\pi$ is an isomorphism over $X^\circ$, and the log structure on $X'$ is the compactifying log structure associated with the strictly normal crossing divisor $D':=(X'\setminus X^\circ)_{\mathrm{red}}$.
	\end{thm}

	One motivation for this article comes from the author's upcoming work \cite{xslogsyn} on logarithmic syntomic cohomology via period sheaves: in that paper, we compute the cohomology of certain period sheaves on the pro-Kummer \'etale site of an fs log adic space log smooth over $C$, and compare it with Hyodo--Kato cohomology. One technical step in \cite{xslogsyn} is to reduce to the case where the log structure comes from a strictly normal crossing divisor. The acyclicity results in this article imply the following corollary.

	\begin{cor}
		Assume that $K$ is perfectoid. Let $X$ be an fs log adic space log smooth over $K$, let $\f\in\{\mathcal O^+/p^m,\widehat{\mathcal O}^+,\mathcal O^{\flat+},\mathbb{A}_{\inf},\mathbb{B}_{\operatorname{dR}},\mathbb{B}_{\operatorname{dR}}^+,\mathbb{B},\cdots\}$ be a period sheaf on $X_{\proket}$, and let $Y \to X$ be a log blow-up along a coherent ideal $\mathcal J\subset\mathcal M_X$ whose geometric stalks are nonempty. Then for $i \geq 0$, the natural map $$H^i_{\proket}(X,\f) \to H^i_{\proket}(Y,\f)$$ is an (almost) isomorphism.
	\end{cor}

	Therefore, by the resolution of toric singularities, whenever we want to compute the cohomology of period sheaves on the pro-Kummer \'etale site of a log smooth fs log rigid analytic variety, we may assume that the log structure comes from a strictly normal crossing divisor\footnote{In fact, the local resolution (Corollary \ref{localresolution}) suffices here.}.

	The proofs for locally constant sheaves and coherent sheaves are similar to those for log schemes in \cite[Theorem 6.2]{illusie2002log} and \cite[Theorem 11.3]{kato1994toric}, respectively; the former also uses a proper base change theorem, which we prove in this article. For $\mathcal O^+/p$, we prove a relative version of the primitive comparison theorem (Theorem \ref{primitive}), and the claim then follows from the result for local systems. The case of $\mathcal O^+$ is then deduced from the cases of $\mathcal O^+/p$ and $\mathcal O$.

	On the other hand, our proof of global resolutions is different from that in \cite{niziol2006toric}. To resolve the singularities of a log regular scheme, the approach in \cite{niziol2006toric} first reduces the problem to log schemes with Zariski log structures, to which one can associate Kato fans, and then resolves the Kato fans. This approach has serious issues in the rigid analytic setting. Roughly speaking, if the log structure is associated with a divisor $D$, then the first step in \cite{niziol2006toric} is to resolve $D$ so that it has no self-intersections. This step does not directly carry over to the rigid analytic setting: the analytic topology is too fine for the same construction of strata and their generic points. In particular, the construction of Kato fans used in that approach is not available here. 
	
	To resolve this issue, we apply a different algorithm, recently developed by Jaros\l{}aw W\l{}odarczyk in \cite{wlodarczyk2022toric}: we first resolve the singularities so that the log structure is associated with a normal crossing divisor, and then separate the branches of this divisor to obtain a strictly normal crossing divisor. The second step relies on the normal crossing assumption, and therefore comes after the resolution of toric singularities. After some modifications, this approach applies to our situation for analytic varieties. We carry it out in Section \ref{globalres}.

\begingroup
\makeatletter
\let\@tocwrite\@gobbletwo
\makeatother

	\subsection*{Acknowledgement} The idea of studying global resolutions in the rigid analytic setting originated during the author's PhD studies under the supervision of Wiesława Nizioł. The author is grateful to her for everything he learned about logarithmic geometry under her guidance. The author also thanks Ben Heuer for many valuable discussions, especially concerning the cohomology of the integral structure sheaf.

	The author acknowledges funding by the Deutsche Forschungsgemeinschaft (DFG, German Research Foundation) through the Collaborative Research Centre TRR 326 ‘Geometry and Arithmetic of Uniformized Structures’, project number 444845124.

	\subsection*{Declaration of AI usage} This article was proofread and polished by Claude Opus 5.5. ChatGPT 5.5 found the reference \cite{wlodarczyk2022toric} during a discussion; this reference plays a key role in the construction of global resolutions in the rigid analytic setting. Discussions with other versions of ChatGPT and ChatGPT Pro also improved this article.

	\subsection*{Notation and conventions}
	
	We adopt the language of adic spaces as developed in \cite{huber2013etale}. For a $p$-adic field $L$, a rigid analytic variety over $L$ is a quasi-separated adic space locally of finite type over $\Spa(L,\mathcal O_L)$. All rigid analytic varieties considered are over $K$ or $C$. We assume all rigid analytic varieties are separated, taut, and countable at infinity.

	We also freely use the language of log adic spaces and toric geometry. For definitions and properties of log adic spaces, we refer the reader to \cite{dllz2023logadic}. 

\par
\endgroup
	
	\section{Preliminaries}
	
	\subsection{Log smoothness}
	
	We recall the definition and some basic properties of log smooth morphisms. The main goal of this subsection is to identify the log structure of an fs log smooth rigid variety with the compactifying log structure determined by its trivial locus. This is a standard result for log regular schemes; see \cite[Theorem 11.6]{kato1994toric} and \cite[Theorem I.1.9.4]{ogus2018log}. The corresponding compactifying log structure in the rigid analytic setting is described in \cite[Example 2.3.16]{dllz2023logadic}. We include the argument for completeness.

	We first recall the definition of log smoothness from \cite[Definition 3.1.1]{dllz2023logadic}.
	
	\begin{defn}
		Let $f : Y \rightarrow X$ be a morphism between locally noetherian fs log adic spaces. We say that $f$ is log smooth if, \'etale locally on $Y$ and $X$, it admits an fs chart $u : P \rightarrow Q$ satisfying the following conditions:
		
		(1) the kernel and the torsion part of the cokernel of $u^{\mathrm{gp}} : P^{\mathrm{gp}} \rightarrow Q^{\mathrm{gp}}$ are finite groups whose orders are invertible in $\mathcal{O}_X$; and
		
		(2) $f$ and $u$ induce a morphism $Y \rightarrow X \times_{X\langle P \rangle} X\langle Q \rangle$ of log adic spaces whose underlying morphism of adic spaces is \'etale.
	\end{defn}

	The following local description is \cite[Proposition 3.1.10]{dllz2023logadic}.

	\begin{prop}
		Let $X$ be an fs log adic space log smooth over $\mathrm{Spa}(K, K^+)$, where $(K, K^+)$ is an affinoid field. Then, \'etale locally on $X$, there exist a sharp fs monoid $P$ and a strictly \'etale morphism $X \rightarrow \mathrm{Spa}(K\langle P \rangle, K^+\langle P \rangle)$ that is a composition of rational localizations and finite \'etale morphisms.
	\end{prop}

	We now prove the main result of this subsection.
		
	\begin{prop}
		Let $(X,\mathcal{M})$ be an fs log rigid variety log smooth over $\Spa(K,K^+)$, where $(K,K^+)$ is an affinoid field. Let $U$ be the maximal open subspace of $X$ on which $\mathcal{M}$ is trivial, and let $j : U \hookrightarrow X$ be the inclusion. Then the structure morphism induces a natural isomorphism $$\alpha : \mathcal{M} \xrightarrow{\simeq} j_*(\mathcal{O}_U^\times) \cap \mathcal{O}_X,$$ where the intersection is taken inside $j_*\mathcal{O}_U$.
	\end{prop}

	\begin{proof}
		Since $(X,\mathcal{M})$ is log smooth over $\Spa(K,K^+)$, \'etale locally there exist a sharp fs monoid $P$ and a strictly \'etale morphism $X \to \Spa(K\langle P \rangle,K^+\langle P \rangle)$. Since the map is strict, $U$ is the pullback of the trivial locus of the log structure of $\Spa(K\langle P \rangle,K^+\langle P \rangle).$



		Let $\overline{x}\to X$ be a geometric point, set $R:=\mathcal{O}_{X_{\et},\overline{x}}$, and let $F:=\{p\in P\mid \alpha(p)\in R^\times\}.$ Then $F$ is a face of $P$, and the canonical chart identifies $$\overline{\mathcal{M}}_{\overline{x}}\simeq P/F \qquad\text{and}\qquad \mathcal{M}_{\overline{x}}\simeq P\oplus_F R^\times.$$ Under the structure morphism, the image of the latter monoid in $R$ consists of the elements $u\alpha(p)$ with $u\in R^\times$ and $p\in P$.

		Since $P$ is saturated, $K\langle P\rangle$ is normal, and hence $R$ is a normal noetherian local domain. Put $Q:=P/F$. The height-one prime ideals $\mathfrak{q}$ of $Q$ correspond to the height-one prime ideals $\mathfrak{p}_{\mathfrak{q}}$ of $R$ lying in the toric boundary. Let $v_{\mathfrak{q}}$ be the normalized valuation of the discrete valuation ring $R_{\mathfrak{p}_{\mathfrak{q}}}$. Since the elements of $F$ map to units in $R$, the restriction of $v_{\mathfrak{q}}$ to the monomials induces a homomorphism $$\nu_{\mathfrak{q}}:Q^{\mathrm{gp}}\longrightarrow\mathbb{Z}.$$ By the local description of Cartier divisors on toric varieties, which is unchanged by strictly \'etale localization, the map $$a\longmapsto\sum_{\operatorname{ht}(\mathfrak{q})=1}\nu_{\mathfrak{q}}(a)[\mathfrak{p}_{\mathfrak{q}}]$$ identifies $Q^{\mathrm{gp}}$ with the group of Cartier divisors supported on the boundary. In particular, this map is injective because $Q$ is sharp.

		Let $f\in R$ represent a germ in $(j_*(\mathcal{O}_U^\times)\cap\mathcal{O}_X)_{\overline{x}}$. Since $f$ is invertible on the trivial locus, its principal Weil divisor is supported on the toric boundary and is effective. Thus the preceding discussion gives an $a\in Q^{\mathrm{gp}}$ such that $$v_{\mathfrak{q}}(f)=\nu_{\mathfrak{q}}(a)\geq 0$$ for every height-one prime $\mathfrak{q}$ of $Q$. Since $Q$ is saturated, we have $$Q=\{a\in Q^{\mathrm{gp}}\mid \nu_{\mathfrak{q}}(a)\geq 0\text{ for every height-one prime }\mathfrak{q}\}.$$ Hence $a\in Q$. Choosing a representative $p\in P$ of $a$, we obtain
		$\operatorname{div}(f)=\operatorname{div}(\alpha(p))$. Thus $f/\alpha(p)$ has trivial divisor. Since $R$ is normal, $f/\alpha(p)\in R^\times$. Hence $f=u\alpha(p)$ for some $u\in R^\times$, proving surjectivity. Conversely, every $u\alpha(p)$ is invertible on $U$. Finally, the uniqueness of $a\in P/F$ shows that two such expressions define the same element of $P\oplus_F R^\times$ whenever they have the same image in $R$, proving injectivity.
	\end{proof}

	\subsection{Log geometric points} \label{loggeompt}
	
	
	We recall the definition and basic properties of log geometric points, which will be used throughout this article. Just as geometric points detect isomorphisms between \'etale sheaves on an adic space, log geometric points detect isomorphisms between Kummer \'etale sheaves on an fs log adic space.
	
	\begin{defn}
		(1) A log geometric point is a log point $\zeta = (\operatorname{Spa}(l, l^+), \mathcal{M}, \alpha)$ satisfying the following conditions:
		\begin{enumerate}
			\item[(a)] $l$ is a complete separably closed nonarchimedean field; and
			\item[(b)] if $M := \Gamma(\operatorname{Spa}(l, l^+), \mathcal{M})$, then $\overline{M} = M/l^\times$ is uniquely $n$-divisible for all positive integers $n$ invertible in $l$.
		\end{enumerate}
		
		(2) Let $X$ be a locally noetherian fs log adic space. A \emph{log geometric point of $X$} is a morphism of log adic spaces $\eta : \zeta \to X$, where $\zeta$ is a log geometric point.
		
		(3) Let $X$ be a locally noetherian fs log adic space. A \emph{Kummer \'etale neighborhood} of a log geometric point $\eta : \zeta \to X$ is a factorization $\zeta \to U \xrightarrow{\ \phi\ } X$ of $\eta$ in which $\phi$ is Kummer \'etale.
	\end{defn}
	
	We explain how to lift an ordinary geometric point $\xi : \operatorname{Spa}(l, l^+) \to X$ to a log geometric point $\widetilde{\xi}$ above it. By \cite[Proposition 2.3.13]{dllz2023logadic}, after replacing $X$ by an \'etale neighborhood, we may assume that $X$ admits a chart modeled on a sharp fs monoid $P$. Thus, as in \cite[Remark 2.3.3]{dllz2023logadic}, there is a strict closed immersion $X \to X\langle P \rangle$.
	
	Equip $\operatorname{Spa}(l, l^+)$ with the log structure $P^{\mathrm{log}}$ associated with the pre-log structure given by the composite $P \to \mathcal{O}_X(X) \to l$. This makes $(\operatorname{Spa}(l, l^+), P^{\mathrm{log}})$ an fs log point with chart $P \to l$, which we continue to denote by $\xi$. For each positive integer $m$, let $P \hookrightarrow \frac{1}{m}P$ be as in \cite[Definition 4.1.5]{dllz2023logadic}, and set
	$$\xi^{(\frac{1}{m})} := (\operatorname{Spa}(l, l^+) \times_{X\langle P \rangle} X\langle \tfrac{1}{m}P \rangle)_{\mathrm{red}},$$
	equipped with the natural log structure modeled on $\frac{1}{m}P$. Taking the reduced subspace ensures that the underlying adic space of $\xi^{(\frac{1}{m})}$ is isomorphic to $\operatorname{Spa}(l, l^+)$. It follows that
	$$\widetilde{\xi} := \varprojlim_{m} \xi^{(\frac{1}{m})}$$
	is a log geometric point above $\xi$, where the inverse limit runs over all positive integers $m$ invertible in $l$. Its underlying adic space is isomorphic to $\operatorname{Spa}(l, l^+)$, and its natural log structure is modeled on
	$$P_{\mathbb{Q}_{\geq 0}} := \varinjlim_{m} \frac{1}{m}P,$$
	where the direct limit is taken over the same set of integers.
	
	The following lemma is the Kummer \'etale analogue of the corresponding statement for geometric points in the \'etale topology.
	
	\begin{lem}
		Let $\zeta \to X$ be a log geometric point of a locally noetherian fs log adic space. Then
		$$\operatorname{Sh}(X_{\mathrm{k\acute{e}t}}) \longrightarrow {\mathrm{Sets}}, \qquad \mathcal{F} \longmapsto \mathcal{F}_{\zeta} := \varinjlim_U \mathcal{F}(U),$$
		where the direct limit runs over the Kummer \'etale neighborhoods $U$ of $\zeta$, is a fiber functor. Moreover, the fiber functors associated with log geometric points form a conservative system.
	\end{lem}

	We will often use the following proposition.

	\begin{prop} \label{geompoint}
		Let $f:Y \to X$ be a quasi-compact quasi-separated morphism between analytic fs log adic spaces, and let $\xi:\Spa(l,l^+) \xrightarrow{i} X$ be a geometric point of $X$, with its log structure pulled back from $X.$ Denote by $g$ the pullback morphism $g:Y\times_X \xi \to \xi$ in the category of fs log adic spaces. Let $\f$ be an abelian sheaf on $Y_{\ket}$. Then for $q\geq 0,$ we have an isomorphism $$i^{-1}R^qf_{\ket,*}\f\simeq R^qg_{\ket,*}(\f|_{Y\times_X \xi}).$$
	\end{prop}

	\begin{proof}
		The proof of \cite[Lemma 4.4.27]{dllz2023logadic} shows that we have an equivalence of topoi $\xi_{\ket}^{\sim} \simeq \varprojlim U_{i,\ket}^{\sim}$, where the limit is over the filtered category of \'etale neighborhoods $i_U:\xi \to U$ in $X$. In particular, every $V \in \xi_{\ket}$ descends to some $V_{U_0} \in U_{0,\ket}$, and for $q\geq 0$ we have an isomorphism $$H^q\!\left((Y\times_{X} V)_{\mathrm{k\acute et}},\mathcal F\right)\simeq\varinjlim_{U \to U_0}H^q\!\left((Y\times_{X}V_{U_0}\times_{U_0}U)_{\mathrm{k\acute et}},\mathcal F\right).$$ After sheafification, these give $R^qg_{\ket,*}(\f|_{Y\times_X \xi})$ and $i^{-1}R^qf_{\ket,*}\f$, respectively. This completes the proof.
	\end{proof}

	\section{Log blow-ups in rigid geometry} \label{logblowup}

	In this section, we define log blow-ups of fs log rigid varieties via algebraic toric models, study their basic properties, and describe them in terms of subdivisions of fans. As an application, we show that toric singularities (i.e., singularities of fs log smooth rigid varieties) can be resolved locally by log blow-ups. We treat global resolutions in Section \ref{globalres}.

	\subsection{Definition and basic properties}

	We first introduce coherent ideals of log structures, along which log blow-ups are defined.

	\begin{defn}
		Let $X$ be an fs log rigid variety. We say that an ideal $\mathcal{J}\subset\mathcal{M}_X$ is coherent if, \'etale locally on $X$, there exist an fs chart $P\to\mathcal{M}_X$ and a finitely generated ideal $I\subset P$ such that $\mathcal{J}=I\mathcal{M}_X$. We say that $\mathcal{J}$ is invertible if it is \'etale locally generated by one element of $\mathcal{M}_X$.
	\end{defn}

	Unless otherwise stated, all ideals considered below are assumed to have nonempty geometric stalks.

	Let $P$ be a toric monoid, i.e., a sharp fs monoid. Put $T_P:=\Spec(K[P])$, equipped with its canonical log structure, and let $T_P^{\an}$ be its analytification over $K$.

	\begin{defn}
		Suppose that $X$ admits a chart $P\to\mathcal{M}_X$ and that $\mathcal{J}=I\mathcal{M}_X$ for a finitely generated ideal $I\subset P$. Let $T_{P,I}\to T_P$ be the (fs) log blow-up along $I$, obtained by saturating the monoid charts of the blow-up of $T_P$ along the monomial ideal $IK[P]$. We define
		$$\operatorname{LogBl}_{\mathcal{J}}(X):=X\times_{T_P^{\an}}T_{P,I}^{\an}.$$
	\end{defn}

	\begin{rem}
		Since the chart morphism $X\to T_P^{\an}$ is strict, the underlying adic space of the fiber product in the above definition is the ordinary fiber product of the underlying adic spaces.
	\end{rem}

	\begin{rem} \label{localblowup}
		Suppose that $I=(a_1,\ldots,a_r)+P$. For $1\leq i\leq r$, put
		$$P_i^0:=P+\sum_{j=1}^r\N(a_j-a_i)\subset P^{\mathrm{gp}},\qquad P_i:=(P_i^0)^{\mathrm{sat}}.$$
		Then $T_{P,I}$ is covered by the affine toric varieties $\Spec(K[P_i])$, and after base change to $X$, their analytifications give charts of $\operatorname{LogBl}_{\mathcal{J}}(X)$. On the $i$-th chart, the pullback of $\mathcal{J}$ is generated by $a_i$. 
	\end{rem}

	\begin{prop} \label{universal}
		Let $f:Y=\operatorname{LogBl}_{\mathcal{J}}(X)\to X$ be as above, and let $h:Z\to X$ be a morphism of fs log rigid varieties such that $\mathcal{J}\mathcal{M}_Z$ is invertible. Then $h$ factors uniquely through $f$, i.e., there exists a unique morphism $Z\to Y$ over $X$.
	\end{prop}

	\begin{proof}
		The problem is \'etale local on $Z$. Since $\mathcal{J}\mathcal{M}_Z$ is invertible, after further \'etale localization, it is generated by the image of one of $a_1,\ldots,a_r$, say $a_i$. Then the images of all differences $a_j-a_i$ lie in $\mathcal{M}_Z$. Hence the map $P^{\mathrm{gp}}\to\mathcal{M}_Z^{\mathrm{gp}}$ carries $P_i^0$ into $\mathcal{M}_Z$, and it also carries $P_i$ into $\mathcal{M}_Z$ since $\mathcal{M}_Z$ is saturated. Applying the structure morphism of $\mathcal{M}_Z$, we obtain a morphism $Z\to\Spec(K[P_i])^{\an}$, which gives the desired factorization. The uniqueness follows from integrality: the images of the differences $a_j-a_i$ are determined by the map on $P$. Therefore, the local factorizations glue. This completes the proof.
	\end{proof}

	By Proposition \ref{universal}, the construction of $\operatorname{LogBl}_{\mathcal{J}}(X)$ is independent of the choice of the chart. Moreover, the toric models are projective, and the invertible ideals $I\mathcal{O}_{T_{P,I}}$ define relatively ample line bundles which are compatible with the descent data. Therefore, by strict \'etale descent, the construction extends to an arbitrary coherent ideal $\mathcal{J}\subset\mathcal{M}_X$.

	\begin{defn}
		Let $X$ be an fs log rigid variety and let $\mathcal{J}\subset\mathcal{M}_X$ be a coherent ideal. We call the resulting morphism $\operatorname{LogBl}_{\mathcal{J}}(X)\to X$ the log blow-up of $X$ along $\mathcal{J}$.
	\end{defn}

	\begin{prop} \label{pullback}
		Let $f:Y=\operatorname{LogBl}_{\mathcal{J}}(X)\to X$ be a log blow-up, and let $h:Z\to X$ be a morphism of fs log rigid varieties. Then there is a canonical isomorphism
		$$Z\times_XY\simeq\operatorname{LogBl}_{\mathcal{J}\mathcal{M}_Z}(Z).$$
	\end{prop}

	\begin{proof}
		Let $W$ be an fs log rigid variety over $Z$. A morphism $W\to Z\times_XY$ over $Z$ is equivalent to a morphism $W\to Y$ over $X$. By Proposition \ref{universal}, the existence of such a morphism is equivalent to the invertibility of $\mathcal{J}\mathcal{M}_W$. Since the right-hand side represents the same functor, this completes the proof.
	\end{proof}

	\begin{prop} \label{logblowupproperties}
		Let $f:Y=\operatorname{LogBl}_{\mathcal{J}}(X)\to X$ be a log blow-up. Then $f$ is proper and log \'etale, and it is a monomorphism in the category of fs log rigid varieties. Moreover, $f$ is an isomorphism over the trivial locus of the log structure of $X$, and $\mathcal{J}\mathcal{M}_Y$ is invertible.
	\end{prop}

	\begin{proof}
		The properness follows from the properness of the algebraic morphism $T_{P,I}\to T_P$, together with analytification and base change. On the charts in Remark \ref{localblowup}, the maps $P\to P_i$ induce isomorphisms on groupifications, so $f$ is log \'etale by the chart criterion. The remaining assertions follow from Proposition \ref{universal}.
	\end{proof}

	\begin{rem}
		The fiber product in Proposition \ref{pullback} is taken in the category of fs log rigid varieties: if $h$ is not strict, the underlying adic space of $Z\times_XY$ need not be the ordinary fiber product. Similarly, for a general $X$, the log blow-up $\operatorname{LogBl}_{\mathcal{J}}(X)$ need not be the normalization of the ordinary blow-up of $X$ along $\mathcal{J}\mathcal{O}_X$. However, this identification holds if $X$ is log smooth over $K$: since normalization commutes with \'etale base change, it can be checked on strictly \'etale toric charts, where saturating the monoid charts normalizes the monomial blow-up.
	\end{rem}

	We record the following consequence of Proposition \ref{pullback}, which is useful for Kummer towers.

	\begin{cor} \label{perfectoidblowuptower}
		Assume that $K$ is perfectoid and $K^+=K^\circ$. Let $P$ be a toric monoid, let $I\subset P$ be a nonempty finitely generated ideal, and let $Y\to\mathbb{E}_P:=\Spa(K\langle P\rangle,K^\circ\langle P\rangle)$ be the log blow-up along $I\mathcal{M}_{\mathbb{E}_P}$. For each positive integer $m$, put
		$$\mathbb{E}_m:=\Spa\left(K\left\langle\tfrac1mP\right\rangle,K^\circ\left\langle\tfrac1mP\right\rangle\right),\qquad Y_m:=Y\times_{\mathbb{E}_P}\mathbb{E}_m,$$
		where $\mathbb{E}_m$ is equipped with the natural log structure modeled on $\frac{1}{m}P$. Then there exist a perfectoid space $Y_\infty$ and compatible morphisms $Y_\infty\to Y_m$ such that $Y_\infty\sim\varprojlim_mY_m$ in the sense of \cite[Definition 7.14]{scholze2012perfectoid}.
	\end{cor}

	\begin{proof}
		Choose generators $a_1,\ldots,a_r$ of $I$, and let $P_i$ be the monoids in Remark \ref{localblowup}. Since the saturation of $\tfrac1mP+\sum_j\N(a_j-a_i)$ is $\tfrac1mP_i$, the toric charts of $Y_m$ are modeled on the monoids $\tfrac1mP_i$. Put
		$$P_{i,\Q}:=\bigcup_{m\geq1}\tfrac1mP_i.$$
		The completed monoid algebras $K\langle P_{i,\Q}\rangle$ are perfectoid by \cite[Lemma 2.2.15]{dllz2023logadic}, and the associated affinoid perfectoid spaces are tilde-limits of the corresponding charts at finite levels. These identifications are compatible with the monomial localizations on the intersections of the charts. Hence the finitely many perfectoid charts glue to the desired perfectoid space $Y_\infty$. This completes the proof.
	\end{proof}

	\subsection{The log adic space associated with a fan}

	We first fix some basic notions in toric geometry.

	\begin{defn}
		Let $N$ be a free abelian group of finite rank, and put $N_\R:=N\otimes_\Z\R$ and $M:=\operatorname{Hom}(N,\Z)$. A strongly convex rational polyhedral cone is a cone of the form $\sigma=\sum_i\R_{\geq0}v_i\subset N_\R$ with finitely many $v_i\in N$ such that $\sigma\cap(-\sigma)=\{0\}$. A fan $\Sigma$ is a finite collection of such cones which is stable under taking faces and such that the intersection of any two cones in $\Sigma$ is a common face. 
	\end{defn}

	We also recall the notion of regularity, which corresponds to the smoothness of toric varieties.

	\begin{defn}
		A cone is called regular if its primitive ray generators form part of a basis of $N$, and a fan is called regular if all its cones are regular.
	\end{defn}

	As in classical toric geometry, one can associate a log adic space with a fan: the following definition comes from \cite[Section 8]{scholze2012perfectoid}.

	\begin{defn}
		Let $\Sigma$ be a fan. For $\sigma\in\Sigma$, put $P_{\sigma}:=\sigma^\vee\cap M$, and denote by $\mathbb{E}_{\sigma}$ the log adic space $\Spa(K\langle P_{\sigma}\rangle,K^+\langle P_{\sigma}\rangle)$, equipped with the log structure associated with the chart $P_{\sigma}$. For a face $\tau\preceq\sigma$, the monoid $P_\tau$ is a localization of $P_{\sigma}$, and the induced morphism $\mathbb{E}_\tau\to\mathbb{E}_{\sigma}$ is the rational open immersion onto the locus where the localized monomials and their inverses are power bounded. We denote by $\mathbb{E}_{\Sigma}$ the log adic space obtained by gluing the $\mathbb{E}_{\sigma}$, $\sigma\in\Sigma$, and by $\mathcal{M}_{\Sigma}$ its log structure induced by the charts.
	\end{defn}

	Let $T_{\Sigma}$ be the toric variety over $K$ associated with $\Sigma$. Then each $\mathbb{E}_{\sigma}$ is an affinoid open subspace of $T_{\sigma}^{\an}$, and these embeddings identify $\mathbb{E}_{\Sigma}$ with an open subspace of $T_{\Sigma}^{\an}$. The following lemma shows that subdivisions in the algebraic setting pass to the rigid analytic setting.

	\begin{lem} \label{toricrestriction}
		Let $\Sigma'\to\Sigma$ be a subdivision. Then
		$$\mathbb{E}_{\Sigma'}=\mathbb{E}_{\Sigma}\times_{T_{\Sigma}^{\an}}T_{\Sigma'}^{\an}.$$
		In particular, the morphism $\mathbb{E}_{\Sigma'}\to\mathbb{E}_{\Sigma}$ is proper. Moreover, if $K^+=K^\circ$, then $\mathbb{E}_{\Sigma}$ is the rigid generic fiber of the formal completion of the toric $K^\circ$-model, and it equals $T_{\Sigma}^{\an}$ when $\Sigma$ is complete.
	\end{lem}

	Next, we recall the relation between monoid ideals and subdivisions.

	\begin{defn}
		Let $\sigma\subset N_\R$ be a strongly convex rational polyhedral cone, and let $I\subset P_{\sigma}$ be an ideal generated by $a_1,\ldots,a_r$. The support function of $I$ is
		$$\varphi_I(v):=\min_i\langle a_i,v\rangle,\qquad v\in\sigma.$$
		For $1\leq i\leq r$, put
		$$\sigma_i:=\{v\in\sigma\mid\langle a_i,v\rangle\leq\langle a_j,v\rangle\text{ for all }j\}.$$
		The maximal cones among the $\sigma_i$, together with all their faces, form a subdivision $\Sigma_I$ of $\sigma$, called the subdivision associated with $I$. Then $\varphi_I$ is a concave piecewise linear function, and the maximal cones of $\Sigma_I$ are its maximal domains of linearity.
	\end{defn}

	\begin{prop} \label{fanlogblowup}
		Let $I\subset P_{\sigma}$ be as above. Then the morphism $\mathbb{E}_{\Sigma_I}\to\mathbb{E}_{\sigma}$ is the log blow-up along $I\mathcal{M}_{\mathbb{E}_{\sigma}}$. Moreover, its underlying morphism is the normalization of the ordinary blow-up along the monomial ideal generated by the $a_i$.
	\end{prop}

	\begin{proof}
		Let $P_i^0$ and $P_i$ be the monoids in Remark \ref{localblowup} for $P=P_{\sigma}$. Then we have
		$$\operatorname{Cone}(P_i^0)^\vee=\{v \in \sigma \mid\langle a_j-a_i,v\rangle \geq 0 \text{ for all }j\}=\sigma_i,$$
		and hence $P_i=\sigma_i^\vee\cap M=P_{\sigma_i}$. Therefore, the log blow-up $T_{P_{\sigma},I}$ in the algebraic setting is $T_{\Sigma_I}$ with the same gluings, and the first statement follows by restricting to $\mathbb{E}_{\Sigma_I}$. The last assertion follows since saturating the monoids normalizes the charts of the monomial blow-up.
	\end{proof}

	\begin{defn}
		We say that a family $(I_{\sigma})_{\sigma \in \Sigma}$ of finitely generated ideals $I_{\sigma} \subset P_{\sigma}$ is compatible if $I_\tau=I_{\sigma}+P_\tau$ for every face $\tau \preceq\sigma$. Such a family defines a coherent ideal $\mathcal{I} \subset \mathcal{M}_{\Sigma}$, and the subdivisions $\Sigma_{I_{\sigma}}$ glue to a subdivision $\Sigma_{\mathcal{I}}$ of $\Sigma$. A subdivision of $\Sigma$ is called coherent if it is obtained in this way.
	\end{defn}

	\begin{ex}
		Suppose that $\sigma=\sum_{i=1}^n\R_{\geq 0}e_i$ is regular, and let $m_1,\ldots,m_n$ be the dual basis. For $v=e_1+\cdots+e_r$, the ideal $(m_1,\ldots,m_r)+P_{\sigma}$ gives the star subdivision at $v$, and the underlying morphism of the corresponding log blow-up is the ordinary blow-up along the closed subspace $V(m_1,\ldots,m_r)$.
	\end{ex}

	Combined with classical results in toric geometry, our construction gives the following local resolution of toric singularities.

	\begin{cor} \label{localresolution}
		Let $X$ be an fs log rigid variety log smooth over $K$. Then, \'etale locally on $X$, there exists a log blow-up $Y\to X$ such that $Y$ is smooth over $K$ and its log structure is the compactifying log structure associated with a strictly normal crossing divisor.
	\end{cor}

	\begin{proof}
		Since the problem is \'etale local on $X$, by \cite[Proposition 3.1.10]{dllz2023logadic}, we may assume that there exist a sharp fs monoid $P$ and a strictly \'etale morphism $X\to\mathbb{E}_{\sigma} \times \mathbb{B}^d$, where $N:=\operatorname{Hom}(P^{\mathrm{gp}},\Z)$, $\sigma:=\operatorname{Hom}(P,\R_{\geq 0})\subset N_\R$, and $\mathbb{B}^d$ is a polydisc equipped with the trivial log structure. By \cite[Proposition 9.8]{kato1994toric}, there exists a coherent subdivision $\Sigma'\to\sigma$ such that $\Sigma' $ is regular. By Propositions \ref{fanlogblowup} and \ref{pullback}, we obtain the associated log blow-up
		$$Y:=X\times_{\mathbb{E}_{\sigma}}\mathbb{E}_{\Sigma'}\longrightarrow X.$$
		For a cone $\tau\in\Sigma'$ of dimension $r$, since $\tau$ is regular, we have $P_\tau\simeq\N^r\oplus\Z^{\operatorname{rk}N-r}$. Hence $\mathbb{E}_{\Sigma'}$ is smooth. Since the induced morphism $Y\to\mathbb{E}_{\Sigma'}\times\mathbb{B}^d$ is the base change of $X\to\mathbb{E}_{\sigma}\times\mathbb{B}^d$, it is strictly \'etale, and hence $Y$ is also smooth and its log structure is the compactifying log structure associated with a strictly normal crossing divisor. This completes the proof.
	\end{proof}

	\begin{rem}
		The regular subdivision chosen in the above proof depends on the local chart, so these local log blow-ups cannot be glued in general. We construct global resolutions in Section \ref{globalres}.
	\end{rem}

	\section{Proper base change}
	
	\subsection{Comparison between Kummer \'etale and \'etale cohomology}
	
	Let $X$ be a locally noetherian fs log adic space, let $n$ be a positive integer invertible in $\mathcal O_X$, and put $\Lambda :=\Z/n\Z$. Let $\epsilon:X_{\ket}\to X_{\et}$ denote the natural morphism of sites. We briefly recall the computation of $R\epsilon_*\Lambda$ from \cite[Lemma 4.4.29]{dllz2023logadic}.
	
	Consider the sequence $$1\to \Lambda(1) \to \mathcal{M}^{\mathrm{gp}}_{X_{\ket}} \xrightarrow{\times n} \mathcal{M}^{\mathrm{gp}}_{X_{\ket}} \to 1$$ on $X_{\ket},$ which is exact by the discussion before \cite[Lemma 4.4.29]{dllz2023logadic}. The boundary map factors through the characteristic monoid and gives a map $$\overline{\mathcal{M}}^{\mathrm{gp}}_{X_{\et}}\otimes\Lambda\to R^1\epsilon_*\Lambda(1).$$ Taking cup products, as in \cite[Lemma 4.4.29]{dllz2023logadic}, induces a natural map $$c:\bigwedge^q\overline{\mathcal{M}}^{\mathrm{gp}}_{X_{\et}}\otimes\Lambda(-q)\to R^q\epsilon_*\Lambda$$ for every $q\geq0$. By the same lemma, we have
	
	\begin{theorem}  \label{kummervsclassical}
		For every $q\geq0$, the map $c$ is an isomorphism $$c:\bigwedge^q \overline{\mathcal{M}}^{\mathrm{gp}}\otimes \Lambda(-q)\xrightarrow{\simeq} R^q\epsilon_*\Lambda.$$
	\end{theorem}

	\begin{cor} \label{kummervsclassicalcoh}
		We have an isomorphism $$\bigwedge^q \overline{\mathcal{M}}^{\mathrm{gp}}\otimes (\mathcal{O}_{X_{\et}}^+/p)(-q) \xrightarrow{\simeq}R^q\epsilon_*(\mathcal{O}_{X_{\ket}}^+/p).$$
	\end{cor}
	
	\begin{proof}
		This follows from Theorem \ref{kummervsclassical} and \cite[Lemma 4.5.8]{dllz2023logadic}, applied with $Y=X_{\ket}$ and $Z=X_{\et}$.
	\end{proof}
	
    \begin{lem} \label{constructible}
    	If $\f \in D_{zc}^{(b)}(X_{\et},\Lambda)$, then $R\epsilon_*\epsilon^*\f \in D_{zc}^{(b)}(X_{\et},\Lambda).$
    \end{lem}
	
	\begin{proof}
		By \cite[Proposition 3.6]{bhatthansen2022constructible}, we may assume that there exists a finite morphism $f:Y \to X$, with the log structure on $Y$ pulled back from $X$, such that $\f=f_{\et,*}\Lambda.$ As $f$ is finite, the stalk computation shows that $f_{\ket,*}\Lambda=\epsilon^*\f.$ Therefore, by Lemma \ref{finitestrict} below, we may replace $X$ by $Y$, and it suffices to show that $R\epsilon_*\epsilon^*\Lambda \in D_{zc}^{(b)}(X_{\et},\Lambda),$ which follows from Theorem \ref{kummervsclassical}.
	\end{proof}

	\begin{lem} \label{finitestrict}
		Let $f:Y\to X$ be a strict finite morphism between analytic fs log adic spaces. Let $\f$ be an abelian sheaf on $Y_{\ket}$. Then for $i\geq 1$ we have $R^if_{\ket,*}\f=0.$
	\end{lem}

	\begin{proof}
		It suffices to check the stalks at all log geometric points $\zeta=(\Spa(l,l^+),\mathcal{M},\alpha) \to X.$ Let $\xi:\Spa(l,l^+) \to X$ be the associated geometric point. Denote by $g$ the pullback morphism $g:Y\times_X \xi \to \xi$ in the category of fs log adic spaces. According to Proposition \ref{geompoint}, it suffices to show that $R^ig_{\ket,*}(\f|_{Y\times_X \xi})=0,$ or equivalently, that $g_{\ket,*}$ is exact. Therefore we may assume that $X=\xi.$ Since $f$ is strict and finite, $Y=\coprod_i \xi_i$ is a finite disjoint union of copies $\xi_i$ of $\xi.$ Let $\zeta_i$ be the pullback of $\zeta$ along $\xi_i \to \xi$. Taking colimits, we obtain $(f_*\f)_{\zeta}\simeq \prod_i \f_{\zeta_i}.$ This completes the proof.
	\end{proof}
	
	\subsection{Proper base change for local systems}
	
	\begin{theorem} \label{properbaselocal}
		Let
		\[
		\begin{tikzcd}
			Y' \arrow[r, "g'"] \arrow[d, "f'"'] & Y \arrow[d, "f"] \\
			X' \arrow[r, "g"]                  & X
		\end{tikzcd}
		\]
		be an fs cartesian square of fs log rigid varieties, where the underlying morphism $\underline{f}$ is proper and $g$ is strict. Then for an abelian torsion sheaf $\f$ on $Y_{\ket}$, the base change map $$g^{*}Rf_*\f \to Rf'_*g'^{*}\f$$ is an isomorphism in any of the following cases:
		
		(1) $p$ is invertible in $\f$;

		(2) $\operatorname{dim.tr}g=0 $ (see \cite[Definition 1.8.4]{huber2013etale});

		(3) $\f$ comes from a Zariski-constructible sheaf on $Y_{\et}$.
	\end{theorem}
	
	\begin{rem}
		For example, case (2) applies when $g$ is a strict closed immersion.
	\end{rem}
	
	\begin{rem}
		One can also define Zariski-constructibility on the Kummer \'etale site $Y_{\ket}$. However, as explained in \cite[4.8]{illusie2002log}, $Rf_*\f$ need not be Zariski-constructible on $X_{\ket}$ for a Zariski-constructible sheaf $\f$ on $Y_{\ket}$.
	\end{rem}
	
	\begin{proof}
	    It suffices to check the stalks at all log geometric points $\zeta=(\Spa(l,l^+),\mathcal{M},\alpha) \to X'.$ Let $\xi:\Spa(l,l^+) \to X'$ be the associated geometric point. As $X$ and $X'$ admit a chart modeled on a sharp fs monoid $P$ up to \'etale localization, by Proposition \ref{geompoint} and the discussion in Section \ref{loggeompt}, we may assume that $X=\Spa(l,l^+)$ with log structure associated with the pre-log structure $P \to l^+$, $X'=\Spa(l',l'^+)$ with log structure associated with the pre-log structure $P \to l^+\to l'^+$, and $\mathcal{M}$ is the log structure of $\zeta$ associated with the pre-log structure $P_{\mathbb{Q}_{\geq 0}} \to l^+ $. For $m\geq 1$, let $$\mathbb{E}_m :=\operatorname{Spa}\left(l\left\langle\frac{1}{m} P\right\rangle, l^+\left\langle\frac{1}{m} P\right\rangle\right), $$ equipped with the natural log structure modeled on $\frac{1}{m}P$. Take a chart $X' \to X \to \mathbb E_1$. Then for any abelian sheaf $\f$ on $Y_{\ket}$, we have $$(R^if_{\ket,*}\f)_{\zeta}=\varinjlim_m H^i_{\ket}((Y\times_{\mathbb E_1} \mathbb E_m),\f)=\varinjlim_m R^i(\epsilon\circ (f\times_{\mathbb E_1} \mathbb E_m)_{\ket})_*\f,$$ where we identify $\f$ with its restriction.
	    
	    Therefore, to prove the theorem, after replacing $f$ by $\epsilon\circ (f\times_{\mathbb E_1} \mathbb E_m),$ we may assume that $X$ and $X'$ have trivial log structures. In this case, we can factor $f=\underline{f}\circ \epsilon$, where $\underline{f}:\underline{Y} \to \underline{X}$ is the underlying map. The classical proper base change theorem holds in all three cases: when $p$ is invertible in $\f$ or $\operatorname{dim.tr}g=0 $, this is \cite[Theorem 4.1.1]{huber2013etale}; when $\f$ comes from a Zariski-constructible sheaf on $Y_{\et}$, this follows from \cite[Theorem 3.15]{bhatthansen2022constructible} and Lemma \ref{constructible}. Therefore, we are reduced to the case where $f=\epsilon.$ In this case, base change follows from the stalk computation in \cite[Lemma 4.4.27]{dllz2023logadic}, since $g$ is strict.
	\end{proof}
	
	\subsection{Proper base change for \texorpdfstring{$\mathcal{O}_X^+/p$}{OX+/p}}
	
	\begin{theorem} \label{properbasestr}
		Let
		$$
		\begin{tikzcd}
			Y' \arrow[r, "g'"] \arrow[d, "f'"'] & Y \arrow[d, "f"] \\
			X' \arrow[r, "g"]                  & X
		\end{tikzcd}
		$$ be an fs cartesian square of fs log rigid varieties, where the underlying morphism $\underline{f}$ is proper and $g$ is strict. Then the base change map of $\mathcal{O}_{X'}^{+}/p$-modules $$g^{*}Rf_*(\mathcal{O}_Y^{+}/p) \to Rf'_*g'^{*}(\mathcal{O}_Y^{+}/p)$$ is an almost isomorphism.
		
	\end{theorem}
	
	\begin{proof}
		As in the proof of Theorem \ref{properbaselocal}, we may reduce to the case where $X=\Spa(l,l^+)$ and $X'=\Spa(l',l'^+)$ have trivial log structures, and then the same argument goes through.
	\end{proof}
	

	\section{Acyclicity of log blow-ups}
	
	In this section, we prove the acyclicity of certain sheaves under log blow-ups. Let $n$ be a positive integer, and put $\Lambda:=\Z/n\Z.$
	
	\subsection{Local systems} The main result of this subsection is the following theorem.
	
	\begin{theorem} \label{acycliclocal}
		Let $X$ be an fs log rigid variety over $K$, and let $f:Y\to X$ be a log blow-up along a coherent ideal $\mathcal{J}\subset\mathcal{M}_X$ whose geometric stalks are nonempty. Then for an abelian torsion sheaf $\f$ on $X_{\ket}$ which comes from a Zariski-constructible sheaf on $X_{\et}$, the adjunction map $$\f \to Rf_*f^*\f$$ is a quasi-isomorphism.
	\end{theorem}
	
	The theorem follows from the following computation.
	
	\begin{prop} \label{locus}
		Let $(X,\mathcal{M})$ be an fs log rigid variety log smooth over $K$. Let $U$ be the maximal open subspace of $X$ on which $\mathcal{M}$ is trivial, and let $j: U\hookrightarrow X$ be the inclusion. Then the natural map $$R\epsilon_*\Lambda \to Rj_*\Lambda$$ is a quasi-isomorphism.
	\end{prop}
	
	\begin{proof}
		The proof of \cite[Lemma 4.6.2]{dllz2023logadic} goes through: since the problem is \'etale local, by \cite[Proposition 3.1.10]{dllz2023logadic}, we may assume that there exists a strictly \'etale morphism $X \to \mathbb E:=\Spa(K\langle P\rangle,K^+\langle P\rangle)$ for some toric monoid $P$, where the log structure on $\mathbb E$ is associated with the pre-log structure $P \to K\langle P\rangle$. We may further assume that $X=\mathbb E.$ Then $X$ is an open subspace of the analytification of the scheme $\Spec(K[P]).$ The proposition then follows from its analogue for schemes \cite[Corollary 7.5]{illusie2002log} and the comparison theorem between schemes and their analytifications \cite[Theorem 3.8.1]{huber2013etale}.
	\end{proof}
	
	\begin{proof}[Proof of Theorem \ref{acycliclocal}]
		The proof of \cite[Theorem 6.2]{illusie2002log} goes through; we include it for completeness. We show that for any log geometric point $\widetilde{x}$ of $X$, the following holds: $$(*)_{\widetilde{x}} \quad \alpha_{\widetilde{x}} : \mathcal{F}_{\widetilde{x}} \longrightarrow (Rf_*f^*\mathcal{F})_{\widetilde{x}}$$ is an isomorphism. We prove $(*)_{\widetilde{x}}$ by induction on $$r(x) = \mathrm{rk}(\overline{\mathcal{M}}^{\mathrm{gp}})_{\overline{x}}$$ (where $x$ (resp. $\overline{x}$) is the point (resp. geometric point) of $X$ that is the image of $\widetilde{x}$). Let $r \in \mathbb{N}$. Assume $(*)_{\widetilde{x}}$ holds for all $(X, \mathcal{J}, \widetilde{x})$ such that $r(x) < r$. Applying the proper base change theorem (Theorem \ref{properbaselocal}) to the base change by the strict map $\overline{x} \to X$, we may assume that the underlying adic space of $X$ is $\operatorname{Spa} (l,l^+)$, with $l$ a complete separably closed nonarchimedean field, and that $X$ is equipped with a chart $X \to \operatorname{Spa} (l\langle P \rangle,l^+\langle P \rangle)$, with $P$ toric, $P^{\mathrm{gp}} \cong \mathbb{Z}^r$ and $\mathcal{J} = I\mathcal{M}_X$. By d\'evissage on $\mathcal{F}$ we reduce to the case $\mathcal{F} = \Lambda$. Now let $X=\Spec(K[P])^{\an}$ be the analytification of the affine toric variety, with vertex $\{0\}$ corresponding to the map $P \to K$ sending $a$ to 0 for $a \neq 0$ and 0 to 1. Endow $X$ with its canonical log structure and let $f : Y \to X$ be the log blow-up of $X$ along $I$. Applying Theorem \ref{properbaselocal}(3) again to the base change by the strict map $\overline{x} \to X$ sending $x$ to $\{0\}$, we are reduced to showing that the stalk of $\alpha : \Lambda \to Rf_*\Lambda$ at a log geometric point $\{0\}^{\sim}$ above $\{0\}$ is an isomorphism. Let $C \in D^+(X_{\ket}, \Lambda)$ be the cone of $\alpha$, defined by the exact triangle $$\Lambda \xrightarrow{\ \alpha\ } Rf_*\Lambda \longrightarrow C \longrightarrow \Lambda[1].$$ We have to show $C_{\{0\}^{\sim}} = 0 .$ By the induction hypothesis, $C_{\zeta} = 0$ for all log geometric points $\zeta$ over a point $x$ of $X$ distinct from $\{0\}$, since $r(x) < r$ for such a point. For $m\geq1$, put $X_m:=\Spec(K[\frac1mP])^{\an}$ and $Y_m:=Y\times_XX_m$, and let $f_m:Y_m\to X_m$ be the pullback of $f$. Write $\{0\}_m$ for the vertex of $X_m$, and let $C_m$ be the pullback of $C$. By Kummer \'etale localization, $C_m$ is the cone of $\Lambda\to Rf_{m,*}\Lambda$, and it is supported on $\{0\}_m$. The restrictions of these Kummer covers to the vertex compute the stalk at $\{0\}^{\sim}$. Therefore, it suffices to prove that, for every $m$, $$R\Gamma(\{0\}_{m,\ket}, C_m) = 0.$$ Since $$R\Gamma(\{0\}_{m,\ket}, C_m) = R\Gamma(X_{m,\ket}, C_m),$$ it suffices to show that the pullback map $$f_m^* : H^q(X_{m,\ket}, \Lambda) \longrightarrow H^q(Y_{m,\ket}, \Lambda)$$ is an isomorphism for all $q$, which follows from Proposition \ref{locus}, since $X_m$ and $Y_m$ have the same trivial locus.
	\end{proof}
	
	\subsection{Integral structure sheaf}
	
	We prove that the Kummer \'etale cohomology of $\mathcal{O}_X^+$ is invariant under log blow-ups (in the almost sense). However, this fails for \'etale cohomology.

	We first consider coherent cohomology of log smooth varieties.
	
	\begin{theorem} \label{acycliccoh}
		Let $X$ be an analytic fs log rigid variety log smooth over $C$, and let $f:Y\to X$ be a log blow-up along a coherent ideal $\mathcal{J}\subset\mathcal{M}_X$ whose geometric stalks are nonempty. Then the adjunction maps $$\mathcal{O}_X \to Rf_*\mathcal{O}_Y, \quad \Omega_X^{\log} \to Rf_*\Omega_Y^{\log}$$ are quasi-isomorphisms.
	\end{theorem}
	
	\begin{proof}
		For the structure sheaf, as in the proof of Proposition \ref{locus}, we may use strictly \'etale toric charts to reduce to the case $X=\mathbb E:=\Spa(C\langle P\rangle,\mathcal O_C\langle P\rangle).$ The assertion then follows from the corresponding result for schemes \cite[Theorem 11.3]{kato1994toric} and GAGA \cite[Proposition 9.4.2]{fujiwarakato2018rigid}. Since $\Omega_X^{\log}$ is locally free and $f$ is log \'etale, we have $f^*\Omega_X^{\log}\simeq\Omega_Y^{\log}$. The assertion for log differentials then follows from the projection formula.
	\end{proof}

	We then consider the sheaf $\mathcal{O}_X^{+}/p$ on the Kummer \'etale site.

	\begin{theorem} \label{acyclicintegral}
		Let $X$ be an analytic fs log rigid variety over $C$, and let $f:Y\to X$ be a log blow-up along a coherent ideal $\mathcal{J}\subset\mathcal{M}_X$ whose geometric stalks are nonempty. Then the adjunction map $$\mathcal{O}_X^{+}/p \to Rf_*(\mathcal{O}_Y^{+}/p)$$ is an almost quasi-isomorphism over the Kummer \'etale site.
	\end{theorem}

	\begin{rem} \label{notacyclic}
		The claim fails on the \'etale site: for example, let $X$ be the analytic affine plane $\mathbb A^2$, equipped with the log structure associated with the coordinate axes, and let $Y$ be the blow-up of $X$ at the origin.
	\end{rem}

	\begin{proof}
		The problem is \'etale local on $X$, so we may choose a chart modeled on a sharp fs monoid $P$ such that $\mathcal J=I\mathcal M_X$. When $X$ is log smooth over $C$, the theorem follows from Theorem \ref{acycliclocal} and the relative primitive comparison theorem (Theorem \ref{primitive}) below. In general, by adding affinoid coordinates, we obtain a strict closed immersion $i:X\hookrightarrow Z:=\mathbb E_P\times\mathbb B^d$. Let $h:W\to Z$ be the log blow-up along $I$. Then $Z$ and $W$ are log smooth over $C$, and $Y=X\times_ZW$ by Proposition \ref{pullback}. Write $i':Y\hookrightarrow W$ for the induced strict closed immersion. Since $\operatorname{dim.tr}i=0$, Theorem \ref{properbaselocal}(2), together with \cite[Lemma 4.5.7]{dllz2023logadic}, gives $$\mathcal O_X^+/p\simeq i^{-1}(\mathcal O_Z^+/p)\simeq^a i^{-1}Rh_*(\mathcal O_W^+/p)\simeq Rf_*i'^{-1}(\mathcal O_W^+/p)\simeq Rf_*(\mathcal O_Y^+/p).$$ This completes the proof.
	\end{proof}

	\begin{theorem} \label{primitive}
		Let $f:Y\to X$ be a proper log smooth morphism of fs log adic spaces log smooth over $\Spa(K,K^+)$. Let $\mathbb{L}$ be an $\F_p$-local system on $Y_{\ket}$. Then for all $i \geq 0$, there is an almost isomorphism of sheaves of almost $\mathcal{O}^+_{X_{\ket}}$-modules $$(R^if_{\ket,*}\mathbb{L}) \otimes \mathcal{O}_{X_{\ket}}^{+}/p \simeq R^if_{\ket,*}(\mathbb{L}\otimes \mathcal{O}_{Y_{\ket}}^{+}/p).$$
	\end{theorem}
	
	\begin{proof}
		Arguing Kummer \'etale locally, we may assume that $\mathbb L$ is the constant sheaf $\F_p$. Write $\underline f:\underline Y \to \underline X$ for the underlying proper morphism, and let $\epsilon_T:T_{\ket}\to\underline T_{\et}$ be the natural morphism of sites for $T=X,Y$. We consider the natural map $$\theta_f:(Rf_{\ket,*}\F_p)\otimes_{\F_p}\mathcal O_{X_{\ket}}^+/p \longrightarrow Rf_{\ket,*}(\mathcal O_{Y_{\ket}}^+/p).$$ Put $A:=R\epsilon_{Y,*}\F_p$. By Lemma \ref{constructible}, we have $A\in D^b_{zc}(\underline Y_{\et},\F_p)$. Applying \cite[Lemma 4.5.8]{dllz2023logadic} to the morphism from $T$ to $\underline T$ (with trivial log structure), we obtain, for complexes $B$ of $\F_p$-sheaves on $T_{\ket}$, $$(R\epsilon_{T,*}B)\otimes_{\F_p}\mathcal O_{\underline T_{\et}}^+/p\xrightarrow{\simeq}R\epsilon_{T,*}(B\otimes_{\F_p}\mathcal O_{T_{\ket}}^+/p).$$ Applying this for $T=X,Y$ and using $\epsilon_X\circ f_{\ket}=\underline f_{\et}\circ \epsilon_Y$, we can identify $R\epsilon_{X,*} \theta_f$ with
		\begin{equation}\label{primitive-etale-comparison}
			(R\underline f_{\et,*}A)\otimes_{\F_p}\mathcal O_{\underline X_{\et}}^+/p\longrightarrow R\underline f_{\et,*}(A\otimes_{\F_p}\mathcal O_{\underline Y_{\et}}^+/p).
		\end{equation}

		We claim that $R\epsilon_{X,*} \theta_f$ is an almost quasi-isomorphism. By \cite[Proposition 3.6]{bhatthansen2022constructible} (which also applies for general $K^+$), it suffices to consider $A=q_{\et,*}\F_p$ for a finite morphism $q:T\to\underline Y$. Since $Rq_{\et,*}\F_p=q_{\et,*}\F_p$, the relative primitive comparison theorem for rigid spaces \cite[Theorem 4.1]{heuer2024primitive}, applied to $q$, gives $$(q_{\et,*}\F_p)\otimes_{\F_p}\mathcal O_{\underline Y_{\et}}^+/p\simeq^a Rq_{\et,*}(\mathcal O_{T_{\et}}^+/p).$$ By compatibility with composition, \eqref{primitive-etale-comparison} is then identified, up to almost quasi-isomorphism, with the comparison map for the proper morphism $\underline f\circ q$. Applying the same theorem to $\underline f\circ q$ proves the claim.

		Finally, let $U\to X$ be a Kummer \'etale morphism with $U$ quasi-compact, and let $f_U:Y\times_XU\to U$ be the base change of $f$, with the fiber product taken in fs log adic spaces. The same argument as above shows that $R\epsilon_{U,*}(\theta_f|_U)$ is an almost quasi-isomorphism. If $\mathcal C$ is the cone of $\theta_f$, it follows that
		$$R\Gamma(U_{\ket},\mathcal C|_U)\simeq R\Gamma(\underline U_{\et},R\epsilon_{U,*}(\mathcal C|_U))\simeq^a0.$$
		These $U$ form a basis of $X_{\ket}$, so sheafifying the cohomology presheaves gives $\mathcal H^i(\mathcal C)^a=0$ for all $i$. Thus $\theta_f$ is an almost quasi-isomorphism. This completes the proof.
	\end{proof}

	We now pass from $\mathcal O_X^+/p$ to $\mathcal O_X^+$ on the Kummer \'etale site.

	\begin{thm} \label{acyclicintegral2}
		Let $X$ be an analytic fs log rigid variety log smooth over $C$, and let $f:Y\to X$ be a log blow-up along a coherent ideal $\mathcal{J}\subset\mathcal{M}_X$ whose geometric stalks are nonempty. Then the adjunction map $$\mathcal{O}_{X}^{+} \to Rf_*(\mathcal{O}_{Y}^{+})$$ is an almost quasi-isomorphism on the Kummer \'etale site.
	\end{thm}

	\begin{rem}
		By Remark \ref{notacyclic}, the theorem fails on the \'etale site.
	\end{rem}

	\begin{proof}
		The assertion is local on $X$, so we may assume that $X$ is quasi-compact and quasi-separated. Since $\mathcal O_X^+$ and $\mathcal O_Y^+$ are $p$-torsion free, the short exact sequences
		$$0\longrightarrow\mathcal O^+/p^{n-1}\xrightarrow{\ p\ }\mathcal O^+/p^n\longrightarrow\mathcal O^+/p\longrightarrow0$$
		and Theorem \ref{acyclicintegral} give, by induction on $n$, almost quasi-isomorphisms
		$$\mathcal O_X^+/p^n\longrightarrow Rf_*(\mathcal O_Y^+/p^n),\qquad n\geq1.$$

		For $Z=X,Y$, we have
		$$\mathcal O_Z/\mathcal O_Z^+=\varinjlim_{n,\,\times p}\mathcal O_Z^+/p^n.$$ Since the higher direct images of $f$ commute with filtered colimits, taking filtered colimits gives an almost quasi-isomorphism
		$$\mathcal O_X/\mathcal O_X^+\longrightarrow Rf_*(\mathcal O_Y/\mathcal O_Y^+).$$
		We then compare, via the adjunction maps, the short exact sequences
		$$0\longrightarrow\mathcal O_Z^+\longrightarrow\mathcal O_Z\longrightarrow\mathcal O_Z/\mathcal O_Z^+\longrightarrow0.$$
		The middle terms are identified by Theorem \ref{acycliccoh}, applied after Kummer \'etale base change; here coherent cohomology on the Kummer \'etale and analytic sites agrees by \cite[Theorem 4.3.7]{dllz2023logadic}. The quotient terms are almost identified by the preceding paragraph. This completes the proof.
	\end{proof}

	\section{Global resolutions of toric singularities} \label{globalres}

	In this section, we prove that a quasi-compact fs log rigid variety log smooth over $K$ admits a global resolution by log blow-ups whose boundary is a strictly normal crossing divisor. The local resolution was established in Corollary \ref{localresolution}; the main issue here is to make the subdivisions on different charts compatible. To this end, we use the canonical subdivisions of \cite{wlodarczyk2022toric}, and descend the associated monoid ideals to coherent ideals of $\mathcal{M}_X$. After obtaining a normal crossing boundary, we separate its branches by a further sequence of log blow-ups.

	\subsection{Canonical subdivisions of fans}

	We keep the notation $P_{\sigma}$, $\mathbb{E}_{\sigma}$ and $\mathbb{E}_{\Sigma}$ from Section \ref{logblowup}. Let $X$ be an fs log rigid variety over $K$. For a geometric point $\overline{x}$ of $X$, the sharp monoid $\overline{\mathcal{M}}_{X,\overline{x}}$ determines the cone $\sigma_{\overline{x}}$ and the lattice $N_{\overline{x}}$ given by
	$$\sigma_{\overline{x}}:=\operatorname{Hom}(\overline{\mathcal{M}}_{X,\overline{x}},\R_{\geq0}),\qquad N_{\overline{x}}:=\operatorname{Hom}(\overline{\mathcal{M}}_{X,\overline{x}}^{\mathrm{gp}},\Z).$$
	On a chart modeled on $P$, localizing $P$ and then passing to the sharpening identifies these cones with faces of $\operatorname{Hom}(P,\R_{\geq0})$. On overlaps of charts, these identifications are induced by the \'etale sheaf $\overline{\mathcal{M}}_X$. In what follows, we only use these local cones and their identifications.

	We first recall the definition of star subdivisions from \cite[Definition 4.3.3]{wlodarczyk2022toric}, which slightly generalizes the classical one for the purpose of functoriality.

	\begin{defn}[{\cite[Definition 4.3.3]{wlodarczyk2022toric}}]
		Let $\Sigma$ be a finite fan. A center of a star subdivision of $\Sigma$ is a finite set $V$ of nonzero primitive lattice vectors in $|\Sigma|$ such that no cone of $\Sigma$ contains two elements of $V$. The star subdivision $V\cdot\Sigma$ of $\Sigma$ at $V$ is obtained by taking the star subdivision at each $v\in V$.
	\end{defn}

	Since the sets of cones affected by different elements of $V$ are disjoint, the subdivision $V\cdot\Sigma$ does not depend on the order in which these star subdivisions are taken. The same terminology applies to finite disjoint unions of fans. We use the following form of W{\l}odarczyk's functorial desingularization theorem.

	\begin{thm}[{\cite[Theorem 4.6.1]{wlodarczyk2022toric}}] \label{canonicalfans}
		Let $\Sigma$ be a finite disjoint union of fans. Then there is a finite sequence $$\Sigma_n\longrightarrow\cdots\longrightarrow\Sigma_0=\Sigma,\qquad \Sigma_{i+1}=V_i\cdot\Sigma_i,$$ such that each $V_i$ is a center of a star subdivision of $\Sigma_i$ and does not meet any regular cone of $\Sigma_i$, and $\Sigma_n$ is regular. Moreover, the centers $V_i$ are functorial with respect to surjective local isomorphisms of cone complexes (i.e., morphisms mapping each cone, together with its induced lattice, isomorphically onto a cone). In particular, the sequence respects identifications between faces; namely, the $i$-th step is the identity on the faces which do not meet $V_i$.
	\end{thm}

	The following lemma shows that each step in Theorem \ref{canonicalfans} is the log blow-up along a coherent ideal. It also provides the compatibility needed to glue on \'etale covers.

	\begin{lem}[{\cite[Lemma 6.3.8]{wlodarczyk2022toric}}] \label{canonicalstarideal}
		Let $\Sigma$ be a finite fan and $V$ a center of a star subdivision of $\Sigma$. For $\sigma\in\Sigma$, $v\in V\cap\sigma$ and an integer $a>0$, put $$J_{\sigma,v,a}:=\{m\in P_{\sigma}\mid\langle m,v\rangle\geq a\},\qquad J_{\sigma,a}:=\prod_{v\in V\cap\sigma}J_{\sigma,v,a},$$ where $J_{\sigma,a}$ is understood to be the unit ideal if $V\cap\sigma=\varnothing$. Then for a sufficiently divisible integer $a>0$, the ideals $J_{\sigma,a}$, $\sigma\in\Sigma$, define a coherent ideal $\mathcal{J}_{V,a}\subset\mathcal{M}_{\Sigma}$ whose log blow-up is $\mathbb{E}_{V\cdot\Sigma}\to\mathbb{E}_{\Sigma}$. Moreover, these ideals are compatible with localization to faces and with lattice isomorphisms preserving $V$.
	\end{lem}

	\begin{proof}
		By \cite[Lemma 6.3.8]{wlodarczyk2022toric} and Proposition \ref{fanlogblowup}, the log blow-up along $J_{\sigma,v,a}$ corresponds to the star subdivision at $v$ for sufficiently divisible $a$. Since $\Sigma$ is finite and each cone contains at most one element of $V$, a common such $a$ gives the required subdivision on every cone.

		For a face $\tau\preceq\sigma$, we have
		$$J_{\sigma,v,a}+P_\tau=\begin{cases}
			J_{\tau,v,a},&v\in\tau,\\
			P_\tau,&v\notin\tau.
		\end{cases}$$
		Indeed, if $v\in\tau$, the elements inverted pair to zero with $v$; otherwise, some $b\in P_\sigma \cap \tau^\perp$ satisfies $\langle b,v\rangle\geq a$ and becomes a unit. Hence $J_{\sigma,a}+P_\tau=J_{\tau,a}$, so these ideals define the required coherent ideal. Since the pairing is invariant under lattice isomorphisms and vanishes on units, the construction is also compatible with lattice isomorphisms and sharpening, and the resulting ideals therefore descend in $\overline{\mathcal{M}}_X$.
	\end{proof}

	\begin{prop} \label{globalncresolution}
		Let $X$ be a quasi-compact fs log rigid variety log smooth over $K$, and let $X^\circ$ be the maximal open subspace of $X$ on which the log structure is trivial. Then there is a finite sequence of log blow-ups $X_r\to\cdots\to X_0=X$, each of which is an isomorphism over $X^\circ$, such that $X_r$ is smooth over $K$ and $D_r:=(X_r\setminus X^\circ)_{\mathrm{red}}$ is a normal crossing divisor. Moreover, the log structure on $X_r$ is the compactifying log structure associated with $D_r$.
	\end{prop}

	\begin{proof}
		The proof is similar to that of \cite[Theorem 6.5.1]{wlodarczyk2022toric}. Choose finitely many strictly \'etale morphisms $U_\alpha\to X$ covering $X$, such that each $U_\alpha$ is affinoid and admits a strictly \'etale toric chart $U_\alpha\to\mathbb{E}_{\sigma_\alpha}$. Applying Theorem \ref{canonicalfans} to the disjoint union of these fans, we obtain a single sequence of star subdivisions. At a geometric point of an overlap, strictness identifies the two characteristic monoids, and hence identifies the corresponding faces together with their lattices. By Theorem \ref{canonicalfans}, the centers are compatible with these identifications at every step.

		We construct the $X_i$ inductively. Suppose that $X_i$ has been constructed. Choose an integer $a_i>0$ as in Lemma \ref{canonicalstarideal} for all the fans at the $i$-th step. By Lemma \ref{canonicalstarideal}, the images in $\overline{\mathcal{M}}_{X_i}$ of the ideals $J_{\sigma,a_i}$ on the charts agree on overlaps, and hence descend to a coherent ideal $\overline{\mathcal{J}}_i \subset \overline{\mathcal{M}}_{X_i}$. Let $\mathcal{J}_i \subset \mathcal{M}_{X_i}$ be its inverse image, and put
		$$X_{i+1}:=\operatorname{LogBl}_{\mathcal{J}_i}(X_i).$$
		By Proposition \ref{pullback}, the pullback of $X_{i+1}\to X_i$ to each chart is the log blow-up associated with the star subdivision. This completes the induction step, and all the identifications on overlaps are preserved. By Theorem \ref{canonicalfans}, all fans become regular after finitely many steps, which gives the smooth $X_r$ with a normal crossing divisor $D_r$. This completes the proof.
	\end{proof}

	\subsection{Transforming an NC divisor into an SNC divisor}

	In this subsection, let $X$ be a quasi-compact smooth rigid variety over $K$, equipped with the compactifying log structure associated with a normal crossing divisor $D$. We show that $D$ can be transformed into a strictly normal crossing divisor by a sequence of log blow-ups. We follow the argument of \cite[6.2]{wlodarczyk2022toric} with some modifications for the rigid analytic setting.

	More precisely, we construct the centers directly as the closed subspaces defined by coherent ideals. These ideals are defined on \'etale charts on which $D$ splits into smooth branches, and descend to $X$ since they agree on overlaps. This avoids the use of strata, whose description in terms of irreducible components requires additional care in the rigid analytic setting.

	We start with the following definition.

	\begin{defn} \label{unresolvedbranches}
		Let $\Delta$ be a reduced divisor and $H$ a strictly normal crossing divisor on $X$ without common components, such that $B:=\Delta\cup H$ is a normal crossing divisor. For a geometric point $\overline{x}$ of $X$, take an \'etale chart around $\overline{x}$ with coordinates such that
		$$\Delta=V(x_1\cdots x_b),\qquad H=V(y_1\cdots y_h),$$
		where all these branches pass through $\overline{x}$. We call $b(\overline{x}):=b$ the number of branches of $\Delta$ at $\overline{x}$, and put $m:=\max_{\overline{x}}b(\overline{x})$, which is finite.
	\end{defn}

	To transform a normal crossing divisor $D$ into a strictly normal crossing divisor, we start with $\Delta=D$ and $H=\varnothing$, and reduce $m$ at each step until $m\leq1$. The following lemma associates a coherent ideal of $\mathcal{M}_B$ with the closed subset of points $\overline{x}$ with $b(\overline{x})=m$.

	\begin{lem} \label{maximumbranchideal}
		In the situation of Definition \ref{unresolvedbranches}, suppose that $m\geq2$, and let $\mathcal{M}_B$ be the compactifying log structure associated with $B$. Then there exists a coherent ideal $\mathcal{J}_m \subset \mathcal{M}_B$ such that the coherent ideal $\mathcal{I}_m:=\mathcal{J}_m\mathcal{O}_X$ defines a smooth closed subspace $C_m\subset X$ of codimension $m$ with the following properties: the geometric points of $C_m$ are those $\overline{x}$ with $b(\overline{x})=m$, and locally at such a point, there is a coordinate system in which $$\mathcal{I}_m=(x_1,\ldots,x_m),\qquad H=V(y_1\cdots y_h).$$
	\end{lem}

	\begin{proof}
		Locally on $X$, choose a chart modeled on a free monoid $P$, and let $e_1,\ldots,e_r$ be the generators of $P$ corresponding to the branches of $\Delta$. Define the monoid ideal $$\overline{J}_m:=\prod_{\substack{I\subset\{1,\ldots,r\}\\|I|=m}}(e_j\mid j\in I),$$ where the empty product is defined to be the unit ideal. At a geometric point, let $A\subset\{1,\ldots,r\}$ be the set of indices of the branches passing through this point. Then all generators $e_j$ with $j\notin A$ become units. Since $|A|\leq m$, every factor of $\overline{J}_m$ becomes the unit ideal if $|A|<m$, whereas if $|A|=m$, the only factor which does not become the unit ideal is $(e_j\mid j\in A)$.

		Hence the stalks of $\overline{J}_m$ depend only on the set of branches of $\Delta$ passing through the point, and these ideals descend to a coherent ideal of $\overline{\mathcal{M}}_B$. Let $\mathcal{J}_m\subset\mathcal{M}_B$ be its inverse image. Then $\mathcal{I}_m=\mathcal{J}_m\mathcal{O}_X$ is the unit ideal away from the closed subset of points with $b(\overline{x})=m$, and near such a point it is given by the equations $x_1=\cdots=x_m=0$. These equations show that $C_m$ is a smooth closed subspace of codimension $m$. This completes the proof.
	\end{proof}

	\begin{lem} \label{branchseparationstep}
		In the situation of Lemma \ref{maximumbranchideal}, let $\pi:X'\to X$ be the log blow-up of $(X,\mathcal{M}_B)$ along $\mathcal{J}_m$. Let $\Delta'$ and $H^{\mathrm{str}}$ be the strict transforms of $\Delta$ and $H$, respectively, and let $E$ be the reduced exceptional divisor. Then $\pi$ is the ordinary blow-up of $X$ along $C_m$, $X'$ is smooth, $H':=H^{\mathrm{str}}\cup E$ is a strictly normal crossing divisor, and $\Delta'\cup H'$ is a normal crossing divisor. Moreover, $\Delta'$ has at most $m-1$ branches at every geometric point of $X'$.
	\end{lem}

	\begin{proof}
		Let $\overline{x}$ be a geometric point of $C_m$, and choose coordinates as in Lemma \ref{maximumbranchideal}. Then locally at $\overline{x}$, we may write the chart monoid and the ideal as
		$$P=\bigoplus_{j=1}^m\N e_j\oplus \bigoplus_{k=1}^h\N f_k,\qquad J=(e_1,\ldots,e_m),$$
		where $e_j$ and $f_k$ correspond to $x_j$ and $y_k$, respectively. For $1\leq i\leq m$, the $i$-th chart of the log blow-up (see Remark \ref{localblowup}) is modeled on the monoid
		$$P_i=P+\sum_{j\ne i}\N(e_j-e_i)=\N e_i\oplus \bigoplus_{j\ne i}\N(e_j-e_i)\oplus \bigoplus_{k=1}^h \N f_k.$$
		Since these generators form a basis of $P^{\mathrm{gp}}$, the monoid $P_i$ is free. On the corresponding chart of $X'$, we have $x_j=x_iu_j$ for $j\ne i$, where $u_j$ is the function corresponding to $e_j-e_i$. These are the charts of the ordinary analytic blow-up of $X$ along $(x_1,\ldots,x_m)$, and the identifications agree on overlaps. Hence $\pi$ is the blow-up of $X$ along $C_m$.

		On this chart, the three parts of the boundary are given by
		$$\Delta'=V\Bigl(\prod_{j \ne i}u_j \Bigr),\qquad E=V(x_i),\qquad H^{\mathrm{str}}=V(y_1\cdots y_h).$$
		Thus $X'$ is smooth and the boundary is a normal crossing divisor. The divisor $E$ is smooth, since it is the projectivization of the normal bundle of $C_m$ in $X$. The same computation shows that the strict transforms of the components of $H$, together with $E$, are smooth, and so are their intersections. Therefore, $H'$ is a strictly normal crossing divisor.

		Finally, above $C_m$, at most $m-1$ branches of $\Delta'$ appear on each chart. Away from $C_m$, the morphism $\pi$ is an isomorphism, and $\Delta$ already has at most $m-1$ branches there. This proves the last assertion.
	\end{proof}

	\begin{prop} \label{nctosnc}
		Let $X$ be a quasi-compact smooth rigid variety over $K$, equipped with the compactifying log structure associated with a normal crossing divisor $D$. Then there is a finite sequence of log blow-ups, with composition $\pi:X'\to X$, such that $(\pi^{-1}D)_{\mathrm{red}}$ is a strictly normal crossing divisor. Moreover, the log structure on $X'$ is the compactifying log structure associated with $(\pi^{-1}D)_{\mathrm{red}}$.
	\end{prop}

	\begin{proof}
		Put $X_0:=X$, $\Delta_0:=D$ and $H_0:=\varnothing$. At each step, let $\Delta_i$ be the strict transform of $D$ on $X_i$, and let $H_i$ be the reduced union of the exceptional divisors created so far. If the maximal number $m_i$ of branches of $\Delta_i$ is at least two, we apply Lemma \ref{branchseparationstep} to $(\Delta_i,H_i)$ to obtain $X_{i+1}$. After at most $\max\{0,m_0-1\}$ steps, we obtain some $X_N$ with $m_N\leq1$. Then $\Delta_N$ is a smooth divisor, and therefore $\Delta_N \cup H_N$ is a strictly normal crossing divisor. Since each step is a log blow-up by construction, this completes the proof.
	\end{proof}

	Combining Propositions \ref{globalncresolution} and \ref{nctosnc}, we obtain the following global resolution theorem.

	\begin{thm} \label{globalresolution}
		Let $X$ be a quasi-compact fs log rigid variety log smooth over $K$, and let $X^\circ$ be the maximal open subspace of $X$ on which the log structure is trivial. Then there exists a morphism $\pi:X'\to X$ which is a composition of log blow-ups, such that $X'$ is smooth over $K$, $\pi$ is an isomorphism over $X^\circ$, and the log structure on $X'$ is the compactifying log structure associated with the strictly normal crossing divisor $D':=(X' \setminus X^\circ)_{\mathrm{red}}$.
	\end{thm}

	\begin{rem}
		In the rigid analytic setting, an argument similar to that of \cite[Corollary 4.11]{niziol2006toric} would show that a composition of log blow-ups is again a log blow-up. In that case, a single log blow-up suffices in Theorem \ref{globalresolution}.
	\end{rem}

	\bibliography{ref}
	
	\bibliographystyle{alpha}
	
\end{document}